\documentclass[12pt,twoside,final,a4paper]{amsart}
\usepackage[margin=1.5cm,includefoot,footskip=30pt]{geometry}
\usepackage{amsaddr}
\usepackage{amssymb}
\usepackage[all,cmtip]{xy}
\usepackage{cite}
\usepackage[utf8]{inputenc}
\usepackage{upgreek}
\usepackage{fix-cm}
\usepackage{graphicx}
\usepackage[T1]{fontenc}
\usepackage{times}
\usepackage{array}
\usepackage{epstopdf}
\usepackage{booktabs}
\usepackage[toc,page]{appendix}
\usepackage[section]{placeins}
\usepackage{longtable}
\long\def\symbolfootnote[#1]#2{\begingroup%
\def\thefootnote{\fnsymbol{footnote}}\footnote[#1]{#2}\endgroup}

\newcommand{\tr}{\ensuremath{{}^T\!\!}}
\newcommand{\tra}{\ensuremath{{}^T}\!}
\newcommand{\C}{\mathfrak C}

\newcommand{\E}{\mathcal E}
\newcommand{\F}{\mathcal F}

\newcommand{\diag}{\textup{diag}}

\newtheorem{theorem}{Theorem}[section]
\newtheorem{lemma}[theorem]{Lemma}

\newtheorem{proposition}[theorem]{Proposition}
\newtheorem{definition}{Definition}[section]
\newtheorem*{theorem*}{Theorem}
\newtheorem*{theorema}{Theorem A}
\newtheorem*{corollarya}{Corollary B}
\newtheorem*{corollaryb}{Corollary C}
\newtheorem{remark}[theorem]{Remark}
\numberwithin{equation}{section}

\newcommand{\ignore}[1]{}

\newcommand{\mynote}[1]{}

\usepackage{mathptmx}
\makeatletter
\def\Ddots{\mathinner{\mkern1mu\raise\p@
\vbox{\kern7\p@\hbox{.}}\mkern2mu
\raise4\p@\hbox{.}\mkern2mu\raise7\p@\hbox{.}\mkern1mu}}
\makeatother

\begin{document}
\title{Gaussian Elimination in symplectic and orthogonal groups}
\author[Bhunia, Mahalanobis, Shinde and Singh]{Sushil Bhunia, Ayan Mahalanobis, Pralhad Shinde and Anupam Singh} 
\address{IISER Pune, Dr.~Homi Bhabha Road, Pashan, Pune 411008, INDIA.} 
\email{sushilbhunia@gmail.com}
\email{ayan.mahalanobis@gmail.com}
\email{anupamk18@gmail.com}
\email{pralhad.shinde96@gmail.com}
\thanks{This work is supported by a SERB research grant.}
\today
\begin{abstract} This paper studies algorithms similar to the Gaussian elimination algorithm in symplectic and orthogonal groups. We discuss two applications of this algorithm. One computes the spinor norm and the other computes the double coset decomposition with respect to Siegel maximal parabolic subgroup.
\end{abstract}
\maketitle
\section{Introduction}
Gaussian elimination is a very old theme in Mathematics. It appeared in print as chapter eight in a Chinese
mathematical text called, ``The nine chapters of the mathematical art''. It is believed, a part of that book was written as early as 150
BCE. For a historical perspective on Gaussian elimination, we refer to a nice work by Grcar~\cite{grc}.   

In this paper, we work with Chevalley generators~\cite[\S11.3]{ca}. Chevalley generators for the special linear group are elementary transvections.  Transvections of the special linear groups have many interesting
applications. Hildebrand~\cite{hildebrand} showed that random walk based on transvections become close to the uniform distribution fast.

Chevalley generators for other classical groups are known for a very long time. However, its use in row-column operations in symplectic and orthogonal groups is new. We develop row-column operations, very similar to the Gaussian elimination algorithm for special linear groups. We call our algorithms Gaussian elimination in symplectic and orthogonal groups respectively. Similar algorithm for unitary groups~\cite{aa} is available.

From our algorithm, one can compute the spinor norm easily, see Section~\ref{spinornorm}. Murray and Roney-Dougal~\cite{mr} studied computing spinor norm earlier. Our algorithm can also be used to compute the double coset decomposition corresponding to the Siegel maximal parabolic subgroup, see Section~\ref{parabolicdecomp}. 

\section{Existing work} We report on some existing work that are relevant to this paper. In computational group theory, Gaussian elimination algorithms are seen as a subprocess of the constructive group recognition project. In this case, a group $G$ is defined by a set of generators $\langle X\rangle=G$, the problem is to write $g\in G$ as a word in $X$. 

Brooksbank~\cite[Section 5]{brooksbank} presents an idea similar to that in our algorithm. His main interest lies in constructive group recognition, writing $g$ as a word in $X$. He is not particularly interested in developing a Gaussian elimination algorithm. In a Gaussian elimination algorithm two important components are: elementary matrices and elementary operations. Brooksbank's elementary matrices are the output of a probabilistic Las Vegas algorithm. So, it is difficult to judge, if he has the same elementary matrices as
ours. He does not define elementary operations. He uses a low dimensional oracle to solve the word problem. Our algorithm is more straightforward and works directly with elementary matrices. It seems that his methods could be modified to produce a Gaussian elimination algorithm in all classical groups in finite fields of all characteristics. However, his treatment depends on the primitive element $\rho$ of $\mathbb{F}_q^\times$ and on expressing $\mathbb{F}_q$ as a finite dimensional vector space over $\mathbb{F}_p$ -- the prime subfield. This suggests that his algorithm would only work for finite fields.

Costi~\cite{costi} develops an algorithm similar to ours using standard generators. These standard generators are defined using a primitive element $\omega$ of the finite field. When one uses the primitive element of the finite field, then to work with an arbitrary field element, one needs to solve the discrete logarithm problem in $\omega$. Discrete logarithm problem in finite fields are hard to solve. There is no known polynomial time algorithm to solve them. This makes programs using the standard generators slower. In the Magma
implementation that we talk about later, for small fields, the field elements are represented by Zech logarithm. So, there is no need to compute the discrete logarithm. However, one needs to compute various
powers of $\omega$ and $\omega^{-1}$ in Costi's algorithm. This is the main reason behind Costi's algorithm to be slower than ours. Furthermore, Costi's algorithms are recursive. Moreover, Costi's algorithm cannot
be extended to infinite fields because of its use of a primitive element.

Cohen, Murray and Taylor~\cite{CMT} proposed a generalized algorithm using the row-column operations, using a representation of Chevalley groups. The key idea there was to bring down an element to a maximal
parabolic subgroup and repeat the process inductively. Here we use the natural matrix representation of these groups. Thus our algorithm is more direct and works with matrices explicitly and effectively. A novelty of our algorithm is that we do not need to assume that the Chevalley generators generate the group under consideration. Thus our algorithm proves independently the fact that these groups are generated by 
elementary matrices.

\section{Main Result}
Algorithms that we develop in this paper work only for a given bilinear form (or quadratic form) $\beta$(see Equations 4.1- 4.5). Though in our algorithm, we work with only one bilinear form (or quadratic form) $\beta$, given by a fixed basis, with a suitable change of basis matrix our algorithm works on any equivalent bilinear forms (or quadratic forms). Our algorithm work well on fields of all characteristics for symplectic and orthogonal groups. 

Another way to look at this paper, we have an algorithmic proof of this well-known theorem.  For definitions of elementary matrices, one can look ahead to Section~\ref{wordproblem}. 
\begin{theorema}\label{thma}
 Let $k$ be a field. For $d\geq 4$ or $l\geq 2$ following holds:
\begin{itemize}
\item[($\mathcal{A}$)] Every element of the orthogonal group $\text{O}(d, k)$ can be written as a product of elementary matrices and a diagonal matrix. Furthermore, the diagonal matrix is of the form 
\begin{eqnarray*}
\diag(1,\ldots,1, \lambda,1,\ldots,1,\lambda^{-1}) & \lambda\in k^\times & \text{for}\; \;O^{+}(2l, k)\\
\diag(\vartheta,1,\ldots,1,\lambda,1,\ldots,1,\lambda^{-1})&\lambda\in k^\times\; \text{and}\;
  \vartheta=\pm 1& \text{for}\; O(2l+1, k)\\
  \diag(1,1,\underbrace{1,\ldots, \lambda}_{(l-1)},\underbrace{1,\ldots, \lambda^{-1}}_{(l-1)}\; ) & \lambda\in k^\times & \text{for}\; O^{-}(2l, k).
\end{eqnarray*}
 \item[($\mathcal{B}$)] Every element of the symplectic group $\text{Sp}(2l, k)$ can be written as a product of elementary matrices.
\end{itemize}
\end{theorema}
This theorem has a surprising corollary. It follows that: 
\begin{corollarya}
Let char$(k)$ be odd. In a split orthogonal group $\text{O}^+(d, k)$, the image of $\lambda$ in $k^\times/k^{\times 2}$ is the spinor norm. 
\end{corollarya}
Furthermore, the spinor norm can also be computed using our algorithm for the twisted orthogonal group in odd characteristics, see Section~\ref{section:twisted_sn}.
\begin{corollaryb}
Let $k$ be a perfect field of characteristics 2. Every element of the orthogonal group $O(2l+1, k)$ can be written as the product of elementary matrices.  
\end{corollaryb}

We have an efficient algorithm to compute the spinor norm. Since the commutator subgroup of the orthogonal group is the kernel of the spinor norm restricted to special orthogonal group, the above corollary is a membership test for the commutator subgroup in the orthogonal group. In other words, an
element $g$ in the special orthogonal group belongs to the commutator subgroup if and only if the $\lambda$ it produces in the Gaussian elimination algorithm is a square in the field, see Equation~\ref{spin}.

The bilinear form that we use and the generators that we define have its roots in the abstract root system of a semisimple Lie algebra and Chevalley groups defined by Chevalley and Steinberg~\cite{st2,ch}. However we assume no knowledge of Lie theory or Chevalley groups in this paper.
  
\section{Orthogonal and Symplectic Groups}\label{des2}
We begin with a brief introduction to orthogonal and symplectic groups. We follow Carter~\cite{ca}, Taylor~\cite{ta} and Grove~\cite{gr} in our introduction. In this section, we \textbf{fix some notations} which will be used throughout this paper. We denote the transpose of a matrix $X$ by $\tr{X}$. As usual, $te_{i,j}$ denote the matrix unit with $t$ in the $(i,j)$ place and $0$ everywhere else.

Let $V$ be a vector space of dimension $d$ over a field $k$. We write the dimension of $V$ as $d$ where $d=2l+1$ or $d=2l$ and $l\geq 1$.  Let $\beta\colon V\times V \rightarrow k$ be
a bilinear form. By fixing a basis of $V$ we can associate a matrix to $\beta$. We shall abuse the notation slightly and denote the matrix of the bilinear form by $\beta$ itself. Thus $\beta(x,y)=\tr x\beta y$ where $x,y$ are column vectors. We will work with non-degenerate bilinear forms, which implies, $\det\beta\neq 0$. A symmetric or skew-symmetric bilinear form $\beta$ satisfies $\beta=\tr\beta$ or $\beta=-\tr\beta$ respectively.  In even order, for characteristic 2, the symmetric and skew-symmetric forms are the same. By fixing a basis for $V$, we identify $\text{GL}(V)$ with $\text{GL}(d, k)$ and treat orthogonal groups and symplectic groups as subgroups of the matrix group $\text{GL}(d, k)$. 

\begin{definition}[Symplectic Group]
A square matrix $X$ of size $d$ is called symplectic if $\tr X\beta X=\beta$ where $\beta$ is skew-symmetric. The set of symplectic matrices form the symplectic group.
\end{definition}
In this paper, we deal with the symplectic group defined by the bilinear form defined by Equation~\ref{beta1}. So any mention of symplectic group means this one particular symplectic group, unless
stated otherwise. 

Up to equivalence, there is an unique non-degenerate skew-symmetric (or alternating) bilinear form over a field $k$. Furthermore a non-degenerate skew-symmetric (or alternating) bilinear form exists only in even dimension. Fix a basis of $V$ as $\{e_1,\ldots, e_l, e_{-1},\ldots, e_{-l}\}$ so that the matrix $\beta$ is: 
\begin{equation}\label{beta1}
\beta=\begin{pmatrix}0&I_l\\ -I_l&0\end{pmatrix}.
\end{equation}
The symplectic group with this $\beta$ is denoted by $\text{Sp}(2l, k)$.

\begin{definition}
Let $Q$ is a quadratic form. The orthogonal group associated with $Q$ is defined as:
\begin{equation*}
\normalfont O(d, Q):=\{X \in GL(V) \,\;|\, \;Q(X(v))=Q(v) \,\, \;\text{for all} \, \,v \in V\}. 
\end{equation*} 
\end{definition}

As the quadratic form is defined in a slightly different way in case of even characteristics we describe the orthogonal groups for odd and even characteristics separately. 
\subsection{Orthogonal groups for char(k)$\neq 2$}  
 Recall that one can easily recover the bilinear form from the quadratic form $Q$ by the formula: 
\begin{equation*}
B(x,y)=\frac{1}{2}\{Q(x+y)-Q(x)-Q(y)\}
\end{equation*}
and it is easy to see that a matrix $X$ satisfies $\tra{X}\beta X = \beta$ if and only if $Q(X(x))=Q(x)$ for all $x\in V$. We work with a particular non-degenerate quadratic form $Q$, however when the characteristics of $k$ is not 2 this corresponds to $\beta$ being non-degenerate. 

Let $k$ be a field of odd characteristic. We work with the following non-degenerate bilinear forms: 
Fix a basis $\{e_0,e_1,\ldots,e_l, e_{-1},\ldots, e_{-l}\}$ for odd dimension and $\{e_1,\ldots, e_l, e_{-1},\ldots, e_{-l}\}$ for even dimension so that the matrix $\beta$ is: 
\begin{equation}\label{beta12}
\beta=\left\{\begin{array}{ll}
\begin{pmatrix}2&0&0\\ 0&0&I_l\\ 0&I_l&0\end{pmatrix} & \text{when}\; d=2l+1\\
\begin{pmatrix}0&I_l\\ I_l&0\end{pmatrix} & \text{when}\; d=2l.
\end{array}\right.
\end{equation}
For the twisted form, we fix a basis $\{e_1, e_{-1}, e_2,\ldots, e_l, e_{-2},\ldots, e_{-l}\}$  so that the matrix $\beta$ is: 
\begin{equation}\label{twisted_beta}
\beta=\begin{pmatrix}
   \beta_{0}&0&0\\
   0&0&I_{l-1}\\
   0&I_{l-1}&0
\end{pmatrix},
\end{equation} where 
$\beta_{0}=\begin{pmatrix}
   1&0\\
   0&\epsilon
\end{pmatrix} $ and $\epsilon$ is a fixed non-square in $k $.\\

Note that if $k(= F_{q} )$ is a finite field of odd characteristic. If $d$ is odd then there is only one orthogonal group up to conjugation \cite[Page 79]{gr} and thus we can fix $\beta$ as above. However, up to conjugation there are two different orthogonal groups \cite[Page 79]{gr} in even dimension $d = 2l$. For orthogonal group with even dimension we fix $\beta$ as above. Thus, we cover all finite orthogonal groups over a field of odd characteristics. 

\subsection{Orthogonal groups for char$(k)=2$}
 Assume that char(k)=2, in this case the quadratic form is defined in a slightly different way. A quadratic  form $Q$ is defined as follows: 
 \begin{equation*}
Q(\lambda x+\mu y)=\lambda^{2}Q(x)+\mu^{2}Q(y)+\lambda \mu B(x, y)
\end{equation*} for all $x,y \in V$, $\lambda, \mu \in k$, and $B$  a symmetric bilinear form on $V$ which is called the associated bilinear form of $Q$.  

 Let $k$ be a field of even characteristics. We work with following non-degenerate quadratic forms \cite[Page 10]{ca}. Fix a basis $\{e_0,e_1,\ldots,e_l, e_{-1},\ldots, e_{-l}\}$ for odd dimension and $\{e_1,\ldots, e_l, e_{-1},\ldots, e_{-l}\}$ for even dimension so that the quadratic forms and their associated bilinear forms are as follows: 
\begin{equation}\label{beta2}
Q(x)=\left\{\begin{array}{ll}
x_{0}^2 +x_{1}x_{-1}+\ldots + x_{l}x_{-l}, & \text{with associated form} \begin{pmatrix}0&0&0\\ 0&0&I_l\\ 0&I_l&0\end{pmatrix}, \quad \text{when}\; d=2l+1\\
x_{1}x_{-1}+\ldots + x_{l}x_{-l},& \text{with associated form} \begin{pmatrix}0&I_l\\ I_l&0\end{pmatrix}, \quad \text{when}\; d=2l.
\end{array}\right.
\end{equation}
For the twisted form, we fix a basis $\{e_1, e_{-1}, e_2,\ldots, e_l, e_{-2},\ldots, e_{-l}\}$  so that the quadratic form and its associated bilinear form is as follows: 
\begin{equation}\label{beta3}
Q(x)=\alpha(x_{1}^{2}+x_{-1}^{2}) +x_{1}x_{-1}+\ldots +x_{l}x_{-l},\quad \text{with the associated form} 
 \begin{pmatrix}
   \beta_{0}&0&0\\
   0&0&I_{l-1}\\
   0&I_{l-1}&0
\end{pmatrix},
\end{equation}
 where $\alpha t^{2}+ t+ \alpha$ is irreducible in $k[x]$ and $\beta_{0}=\begin{pmatrix}
   0&1\\
   1&0
\end{pmatrix} $. 

 Note that if $k$ is a perfect field of even characteristics. It is well known that if $\text{dim}(V)=2l+1$ then there is only one non-degenerate quadratic form upto equivalence. However, in case of $\text{dim}(V)=2l$ there are two quadratic forms upto equivalence. Thus, we cover all orthogonal groups over a perfect field of even characteristics. 

Let $Q$ be a non-degenerate quadratic form defined as above. For a fixed basis, we note that any isometry  $g$ satisfying $Q(g(v))=Q(v)$ for all $v \in V$ also satisfies $\tra{g}\beta g=\beta$. However the converse is not true. We denote othogonal groups associated with $Q$ by $O(2l+1,k)$, $O^{+}(2l, k)$ and $O^{-}(2l, k)$ respectively.  
  
Let $k$ be a field of odd characteristics. We denote by $\Omega(d,k)$ the commutator subgroup of the orthogonal group $\text{O}(d,k)$ which is equal to the commutator subgroup of $\text{SO}(d,k)$. 
There is a well known exact sequence 
\begin{equation}\label{spin}
1\longrightarrow \Omega(d,k)\longrightarrow
\text{SO}(d,k)\overset{\Theta}{\longrightarrow}
k^{\times}/{k^{\times}}^2\longrightarrow 1
\end{equation}
where $\Theta$ is the spinor norm. The spinor norm is defined as
$\Theta(g)=\underset{i=1}{\overset{m}{\prod}} Q(v_i)$ where
$g=\rho_{v_1}\cdots\rho_{v_m}$ is written as a product of reflections.

\section{Solving the word problem in $G$}\label{wordproblem}
In computational group theory, one is always looking for algorithms that solve the word problem. When $G$ is a special linear group, one has a well-known algorithm to solve the word problem -- the Gaussian elimination. One observes that the effect of multiplying an element of the special linear group by an elementary matrix (also known as elementary transvection) from left or right is either a row or a column operation respectively. Using this algorithm one can start with any matrix $g\in\text{SL}(l+1, k)$ and get to the identity matrix thus writing $g$ as a product of elementary matrices~\cite[Proposition 6.2]{ab}. One of the \textbf{objective of this paper} is to discuss a similar algorithm for orthogonal and symplectic groups, with a set of generators that we will call \textbf{elementary matrices in their 
respective groups}.  

We first describe the elementary matrices and the row-column operations for the respective groups. These row-column operations are nothing but multiplication by elementary matrices from left and right respectively. Here elementary matrices used are nothing but Chevalley generators which follows from the theory of Chevalley groups.

The basic idea of the algorithm is to use the fact that multiplying any orthogonal matrix by any one of the generators enables us to perform row or column operations.  
The relation $\tra{g}\beta g=\beta$ gives us some compact relations among the blocks of $g$ which can be used to make the algorithm more faster. To make the algorithm simple we will write the algorithm for $O(2l+1, k)$, $O^{+}(2l, k)$ and $O^{-}(2l, k)$ separately.

\subsection{Groups in which Gaussian elimination works}
\begin{itemize}
\item Symplectic groups: Since all non-degenerate skew-symmetric bilinear forms are equivalent~\cite[Corollary 2.12]{gr}, we have a Gaussian elimination algorithm for all symplectic groups over an arbitrary field.
\item Orthogonal groups:
\begin{itemize}
\item[--] Since non-degenerate symmetric bilinear forms over a finite field of odd characteristics are classified~\cite[Page 79]{gr} according to the $\beta$( see Equations~\ref{beta12} \&~\ref{twisted_beta}), we have a Gaussian elimination algorithm for all orthogonal groups over a finite field of odd characteristics. 
\item[--]  Since non-degenerate quadratic forms over a perfect field of even characteristics can be classified~\cite[Page 10]{ca} according to quadratic forms $Q(x)$ defined in Equations~\ref{beta2} \&~\ref{beta3}, we have a Gaussian elimination algorithm for all orthogonal groups over a perfect field of even characteristics.
\item [--] Furthermore, we have Gaussian elimination algorithm for orthogonal groups that are given by the above bilinear forms or quadratic forms over arbitrary fields. This algorithm also works for bilinear or quadratic forms that are equivalent to the above forms. 
\end{itemize}
\end{itemize}

\subsection{Gaussian elimination for matrices of even size -- orthogonal group $O^{+}(d, k)$ and symplectic groups}
Recall that the bilinear forms $\beta$ are the following: 
\begin{itemize}
\item For symplectic group, Sp$(d, k)$, $d=2l$ and  $\beta=\begin{pmatrix}0&I_l\\ -I_l&0\end{pmatrix}$.
\item For orthogonal group, $O^{+}(d, k)$, $d=2l$ and  $\beta=\begin{pmatrix}0&I_l\\ I_l&0\end{pmatrix}$.
\end{itemize}
Note that any isometry $g$ satisfying the quadratic form $Q$ also  satisfy $\tra{g}\beta g=\beta $. The main reason our algorithm works is the following:
Recall that a matrix $g=\begin{pmatrix}A&B\\C&D\end{pmatrix}$, where $A,B,C$ \& $D$ are matrices of size $l$, is orthogonal or symplectic if $\tr g\beta g=\beta$ for the respective $\beta$. After some usual calculations, for orthogonal group it becomes
\begin{equation}\label{imp1}
\begin{pmatrix}\tr CA+\tr AC & \tr CB+\tr AD\\
               \tr DA+\tr BC & \tr DB+\tr BD
\end{pmatrix}
=
\begin{pmatrix}
0&I_l\\
I_l&0
\end{pmatrix}
\end{equation}
The above equation implies among other things, $\tr CA+\tr AC=0$. This implies that $\tr AC$ is skew-symmetric. In an almost identical way one can
 show, if $g$ is symplectic, $\tr AC$ is symmetric. The working principle of our algorithm is simple -- use the symmetry of $\tr AC$. The problem is, for arbitrary $A$ and $C$, it is not easy to use this symmetry. In our case we were able to reduce $A$ to a diagonal matrix and then it is relatively straightforward to use this symmetry. We will explain the algorithm in details later. First of all, let us describe the elementary matrices and the row-column operations for orthogonal and symplectic groups. The genesis of these elementary matrices lie in the Chevalley basis of simple Lie algebras. We won't go into details of Chevalley's theory in this paper. Furthermore, we don't need to, the algorithm that we produce will show that these elementary matrices are 
generators for the respective groups.

Next we present the elementary matrices for the respective groups and then the row-column operations in a tabular form.

\subsubsection{Elementary matrices (Chevalley generators) for orthogonal group $O^{+}(d, k)$ of even size}

Following the theory of root system in a simple Lie algebra, we index rows by $1,2,\ldots,l,-1,-2,\ldots,-l$. For $t\in k$, the elementary matrices are defined as follows:
\begin{table}[ht]
\caption{Elementary matrices for $O^{+}(2l, k)$} 
\centering 
\begin{tabular}{| c | c | c c |} 
\hline 
&&&\\
  char$(k)$ & & Elementary matrices& \\ [0.5ex]
 \hline 
 &  $x_{i,j}(t)$&$I+t(e_{i,j}- e_{-j,-i})$, & $i\neq j$\\
both  & $x_{i,-j}(t)$&$I+t(e_{i,-j}- e_{j,-i})$,& $i<j$\\
  & $x_{-i,j}(t)$&$I+t(e_{-i,j}-e_{-j,i})$, & $i<j$\\
  & $w_{i}$&$I-e_{i,i}-e_{-i,-i}+e_{i,-i}+e_{-i,i}$& $1\leq i \leq l$ \\
  \hline
\end{tabular}
\label{table:nonlin} 
\end{table}  

Let us note the effect of multiplying $g$ by elementary matrices. We write $g \in O^{+}(2l, k)$ as 
$g=\begin{pmatrix} A&B\\C&D\end{pmatrix}$, where $A,B,C$ \& $D$ are $l \times l$ matrices.
\begin{table}[ht]
\caption{The row-column operations for $O^{+}(2l, k)$} \label{Table 1} \vskip 0.2 cm
\centering 
\begin{tabular}{|c c||c c |} 
\hline 
&&&\\
& Row operations & &Column operations  \\ [0.5ex]
 \hline 
 ER1 & $i^{\text{th}} \mapsto i^{\text{th}}+t j^{\text{th}}$ row and & EC1 & $j^{\text{th}} \mapsto j^{\text{th}}+t i^{\text{th}}$ column and \\
& $-j^{\text{th}} \mapsto -j^{\text{th}}-t (-i)^{\text{th}}$  row & &$-i^{\text{th}} \mapsto -i^{\text{th}}-t(-j)^{\text{th}}$  column\\
 \hline
ER2 & $i^{\text{th}} \mapsto i^{\text{th}}+t (-j)^{\text{th}}$ row and & EC2 & $-i^{\text{th}} \mapsto -i^{\text{th}}-tj^{\text{th}}$ column and \\
&$j^{\text{th}} \mapsto j^{\text{th}}-t (-i)^{\text{th}}$  row & &$-j^{\text{th}} \mapsto -j^{\text{th}}+ti^{\text{th}}$  column\\
   \hline
ER3 &$-i^{\text{th}} \mapsto -i^{\text{th}}-tj^{\text{th}}$ row and & EC3 & $j^{\text{th}} \mapsto j^{\text{th}}+t(-i)^{\text{th}}$ column and \\
 &$-j^{\text{th}} \mapsto -j^{\text{th}}+ti^{\text{th}}$  row & & $i^{\text{th}} \mapsto i^{\text{th}}-t(-j)^{\text{th}}$  column\\
 \hline
 $ w_{i}$ & Interchange $i^{th}$ and $(-i)^{th}$ row && Interchange $i^{th}$ and $(-i)^{th}$ column \\
 \hline
 \end{tabular}
\end{table}
  
\subsubsection{Elementary matrices (Chevalley generators) for symplectic groups}
For $t\in k$, the elementary matrices are defined as follows:

\begin{table}[ht]
\caption{Elementary matrices for $Sp(2l, k)$} 
\centering 
\begin{tabular}{| c | c | c c |} 
\hline 
&&&\\
  char$(k)$ & & Elementary matrices& \\ [0.5ex]
 \hline 
 &  $x_{i,j}(t)$&$I+t(e_{i,j}- e_{-j,-i})$, & $i\neq j$\\
both  & $x_{i,-j}(t)$&$I+t(e_{i,-j}- e_{j,-i})$,& $i<j$\\
  & $x_{-i,j}(t)$&$I+t(e_{-i,j}-e_{-j,i})$, & $i<j$\\
  & $x_{i,-i}(t)$&$I+te_{i,-i}$,& $1\leq i\leq l$ \\
  & $x_{-i,i}(t)$&$I+te_{-i,i}$,& $1\leq i\leq l$ \\
  \hline
\end{tabular}
\label{table:nonlin} 
\end{table}  

Let us note the effect of multiplying $g$ by elementary matrices. We write $g \in Sp(2l, k)$ as 
$g=\begin{pmatrix} A&B\\C&D\end{pmatrix}$, where $A,B,C$ \& $D$ are $l \times l$ matrices.

\begin{table}[h!]
\begin{center}
\caption{The row-column operations for symplectic groups}
\label{tab:symp}
\begin{tabular}{|cc|cc|}
\hline
&&&\\
&Row operations && Column operations\\
\hline
ER1&$i\textsuperscript{th}\mapsto i\textsuperscript{th} + tj\textsuperscript{th}$ row and 
 &EC1& $j\textsuperscript{th}\mapsto j\textsuperscript{th} + ti\textsuperscript{th}$ column and\\ &$-j\textsuperscript{th}\mapsto -j\textsuperscript{th} +t(-i)\textsuperscript{th}$ row && $-i\textsuperscript{th}\mapsto -i\textsuperscript{th} + t(-j)\textsuperscript{th}$ column \\
\hline
ER2&$i\textsuperscript{th}\mapsto i\textsuperscript{th} + t(-j)\textsuperscript{th}$ row and 
&EC2 & $-i\textsuperscript{th}\mapsto -i\textsuperscript{th} + tj\textsuperscript{th}$ column and\\ & $j\textsuperscript{th}\mapsto j\textsuperscript{th} +t(-i)\textsuperscript{th}$ row && $-j\textsuperscript{th}\mapsto -j\textsuperscript{th} + ti\textsuperscript{th}$ column\\
\hline
ER3&$-i\textsuperscript{th}\mapsto -i\textsuperscript{th} + tj\textsuperscript{th}$ row and 
 &EC3& $j\textsuperscript{th}\mapsto j\textsuperscript{th} + t(-i)\textsuperscript{th}$ column and\\ & $-j\textsuperscript{th}\mapsto -j\textsuperscript{th} +ti\textsuperscript{th}$ row && $i\textsuperscript{th}\mapsto i\textsuperscript{th} + t(-j)\textsuperscript{th}$ column\\
\hline
ER1a&$i\textsuperscript{th}\mapsto i\textsuperscript{th} + t(-i)\textsuperscript{th}$ row &EC1a& $-i\textsuperscript{th}\mapsto -i\textsuperscript{th}+ti\textsuperscript{th}$ column\\
\hline
ER2a&$-i\textsuperscript{th}\mapsto -i\textsuperscript{th} + ti\textsuperscript{th}$ row &EC2a& $i\textsuperscript{th}\mapsto i\textsuperscript{th}+t(-i)\textsuperscript{th}$ column\\
\hline
$w_i$&Interchange $i\textsuperscript{th}$ and $(-i)\textsuperscript{th}$ rows && Interchange $i\textsuperscript{th}$ and $(-i)\textsuperscript{th}$ columns\\ &with a sign change in the $i\textsuperscript{th}$ row. && with a sign change in the $i\textsuperscript{th}$ column.\\
\bottomrule
\end{tabular}
\end{center}
\end{table}

\subsubsection{Gaussian elimination for Sp$(2l, k)$ and $O^{+}(2l, k)$.}\label{algo:1}
\begin{description}
\item[Step 1] Use ER1 and EC1 to make $A$ into a diagonal matrix. This makes $A$ into a diagonal matrix and changes other matrices $A$, $B$, $C$ and $D$. For sake of notational convenience we keep calling these changed matrix as $A$, $B$, $C$ and $D$ as well.
\item[Step 2] There are two possibilities. One, the diagonal matrix $A$ is of full rank and two, the diagonal matrix $A$ is of rank $\mathfrak{r}$ less than $l$. This is clearly identifiable by looking for zeros in the diagonal of $A$. 
\item[Step 3] Make $\mathfrak{r}$ rows of $C$, corresponding to the non-zero entries in the diagonal of $A$ zero by using ER3. If $\mathfrak{r}=l$, we have $C$ a zero matrix. If not let us assume that $i\textsuperscript{th}$ row is zero in $A$. Then we interchange the $i\textsuperscript{th}$ row with the $-i\textsuperscript{th}$ row in $g$. We do this for all zero rows in $A$.
 The new $C$ is a zero matrix. We claim that the new $A$ must have a full rank. This follows from Equation~\ref{imp1}; in particular $\tr CB+\tr AD=I_l$. If $C$ is zero matrix then $A$ is invertible. Now make $A$ a diagonal matrix by using Step 1. Then one can make $A$ a matrix of the form 
$\begin{pmatrix}
1 &0&\cdots&0\\
0&1&\cdots&0\\
\vdots&\vdots&\ddots&\vdots\\
0&0&\cdots&\lambda
\end{pmatrix}$
 where $\lambda$ is a non-zero scalar using ER1~\cite[Proposition 6.2]{ab}.
Once $A$ is diagonal and $C$ a zero matrix the equation $\tr CB+\tr AD=I_l$ makes $D$ a diagonal matrix of full rank.
\item[Step 4] Use ER2 to make $B$ a zero matrix. The matrix $g$ becomes a diagonal matrix of the form 
 $\begin{pmatrix}
A &0\\
0&A^{-1}\\
\end{pmatrix}$ where $A$ is of the form $\begin{pmatrix}
1 &0&\cdots&0\\
0&1&\cdots&0\\
\vdots&\vdots&\ddots&\vdots\\
0&0&\cdots&\lambda
\end{pmatrix}$.
\item[Step 5] (only for symplectic groups) Reduce the $\lambda$ to $1$ using Lemma~\ref{lemmap}.
\end{description}
\begin{lemma}\label{lemmap}
For $\textrm{Sp}(2l,k)$, the element $\diag(1,\ldots,1,\lambda,1,\ldots,1,\lambda^{-1})$ is a product of elementary matrices.
\end{lemma}
\begin{proof}
Observe that $(I+te_{l,-l})(I-t^{-1}e_{-l,l})(I+te_{l,-l}) = I-e_{l,l}-e_{-l,-l}+te_{l,-l}-t^{-1}e_{-l,l}$ and denote it by $w_l(\lambda)$ and then the diagonal element is $w_l(\lambda)w_l(-1)$.
\end{proof}
\begin{remark}
As we saw in the above algorithm, we will have to interchange $i\textsuperscript{th}$ and $-i\textsuperscript{th}$ rows for $i=1,2,\ldots,l$. This
 can be done by premultiplying with a suitable matrix.
\end{remark}

Let $I$ be the $2l\times 2l$ identity matrix over $k$. To swap $i\textsuperscript{th}$ and $-i\textsuperscript{th}$ row in O$^{+}(2l, k)$, 
swap $i\textsuperscript{th}$ and $-i\textsuperscript{th}$ rows in the matrix $I$. We will call this matrix $w_i$. It is easy to see that this 
matrix $w_i$ is in O$^{+}(2l, k)$ and is of determinant $-1$. Premultiplying with $w_i$ does the row interchange we are looking for.

In the case of symplectic group Sp$(2l,k)$, we again swap two rows $i^\textsuperscript{th}$ and $-i\textsuperscript{th}$ rows in $I$. However we do a sign change in the $i\textsuperscript{th}$ row and call it $w_i$. Simple computation with our chosen $\beta$ shows that the above matrices are in O$^{+}(2l, k)$ and Sp$(2l, k)$ respectively.

However there is one difference between orthogonal and symplectic groups. In symplectic group, $w_i$ can be generated by elementary matrices because  $w_{i}= x_{i,-i}(1)x_{-i, i}(-1)x_{i,-i}(1)$. In the case of orthogonal groups that is not the case. This is clear, the elementary matrices come from the Chevalley generators and those generates $\Omega$ the commutator of the orthogonal group. All matrices in $\Omega$ have determinant $1$. However $w_i$ has determinant $-1$. So we must add $w_i$ as a elementary matrix for O$^{+}(2l, k)$.   
\begin{remark}
This algorithm proves, every element in the symplectic group is of determinant 1. Note, the elementary matrices for the symplectic group is of determinant 1, and we have an algorithm to write any element as product of elementary matrices. So this proves that the determinant is 1.
\end{remark}
\begin{remark}
This algorithm proves, if $X$ is an element of an symplectic group then so is $\tr X$. The argument is similar to above, here we note that the transpose of an elementary matrix in symplectic groups is an elementary matrix.
\end{remark}

\subsection{Gaussian elimination for matrices of odd size -- the odd-orthogonal group} In this case, matrices are of odd size and there is only one family of group to consider, it is the odd orthogonal group $O(2l+1,k)$. This group will be referred to as the odd-orthogonal group.
\subsection{Elementary Matrices (Chevalley generators) for $O(2l+1,k)$}
Following the theory of Lie algebra, we index rows by $0, 1, \ldots,l,-1,\ldots,-l$. These elementary matrices are listed in Table 5. 

\begin{table}[ht]
\caption{Elementary matrices for $O(2l+1,k)$} 
\centering 
\begin{tabular}{| c | c | c c |} 
\hline 
&&&\\
  char(k) & & Elementary matrices& \\ [0.5ex]
 \hline 
 &  $x_{i, j}(t)$&$I+t(e_{i, j}- e_{-j,-i})$, &  $i\neq j$\\
both  & $x_{i,-j}(t)$&$I+t(e_{i,-j}- e_{j,-i})$,&  $i<j$\\
  & $x_{-i, j}(t)$&$I+t(e_{-i, j}-e_{-j, i})$, &  $i<j$\\
  \hline
 & $ x_{i,0}(t)$&$I+t(2e_{i,0}-e_{0, -i})-t^{2}e_{i,-i}$&  $1\leq i\leq l$\\
odd& $x_{0,i}(t)$&$I+t(-2e_{-i,0}+e_{0, i})-t^{2}e_{-i,i}$& $1\leq i\leq l$ \\
\hline
& $x_{i,0}(t)$&$I+te_{0, -i}+t^{2}e_{i,-i}$& $1\leq i\leq l$\\
even&$x_{0, i}(t)$&$I+te_{0, i}+t^{2}e_{-i, i}$& $1\leq i\leq l$\\
\hline
\end{tabular}
\label{table:elementary_odd} 
\end{table}
Elementary matrices for the odd-orthogonal group in even characteristics differs from that of odd characteristics. In above table we made that distinction and listed them separately in different rows according to the characteristics of $k$.  If char$(k)$ even, we can construct the elements $w_{i}$, which interchanges the $i^{th}$ row with $-i^{th}$ row as follows
\begin{equation*}
 w_{i}=(I+e_{0,i}+e_{-i, i})(I+e_{0,-i}+e_{i,-i})(I+e_{0,i}+e_{-i, i})
 =I+e_{i, i}+e_{-i, -i}+e_{i,- i}+e_{-i, i}.
\end{equation*} 
Otherwise, we can construct $w_{i}$, which interchanges the $i^{th}$ row with $-i^{th}$ with a sign change in $i^{th}, -i^{th} \;\text{and} \; 0^{th}$ row in odd-orthogonal group as follows:
\begin{equation*}
w_{i}=x_{0,i}(-1)x_{i,0}(1)x_{0,i}(-1)=I-2e_{0,0}-e_{i, i}-e_{-i,-i}-e_{i,-i}-e_{-i, i}. 
\end{equation*}
The Gaussian elimination algorithm for  $O(2l+1,k)$ follows the earlier algorithm for symplectic and even-orthogonal group closely, except that we need to take care of the zero row and the zero column. We write an element $g\in O(2l+1,k)$ as $g=\begin{pmatrix}\alpha&X&Y\\ E& A&B\\F&C&D\end{pmatrix}$ where $A,B,C$ \& $D$ are $l\times l$ matrices, $X$ and $Y$ are $1\times l$ matrices, $E$ and $F$ are $l\times 1$ matrices, $\alpha\in k$ and $\beta=\begin{pmatrix}2&0&0\\ 0&0&I_l\\0&I_l&0\end{pmatrix}$. Then from the condition $\tr g\beta g=\beta$ we get the following equations.
\begin{eqnarray}
2\tr XX+\tr AC+\tr CA=0\\
2\alpha\tr X+\tr AF+\tr CE=0\\
2\alpha Y+\tr ED+\tr FB=0\\
2\tr XY+\tr AD+\tr CB=I_l
\end{eqnarray}
Let us note the effect of multiplying $g$ by elementary matrices. 
\begin{table}[ht]
\caption{The row-column operations for $O(2l+1,k)$} \label{Table2} 
\centering 
\begin{tabular}{|c c ||c c|} 
\hline 
& & &\\
& Row operations & &Column operations  \\ [0.5ex]
 \hline 
 ER1 & $i^{\text{th}} \mapsto i^{\text{th}}+t j^{\text{th}}$ row and & EC1 & $j^{\text{th}} \mapsto j^{\text{th}}+t i^{\text{th}}$ column and \\
{\tiny(both)}& $-j^{\text{th}} \mapsto -j^{\text{th}}-t (-i)^{\text{th}}$  row & {\tiny(both)} &$-i^{\text{th}} \mapsto -i^{\text{th}}-t(-j)^{\text{th}}$  column\\
 \hline
ER2 & $i^{\text{th}} \mapsto i^{\text{th}}+t (-j)^{\text{th}}$ row and & EC2 & $-i^{\text{th}} \mapsto -i^{\text{th}}-tj^{\text{th}}$ column and \\
{\tiny(both)} &$j^{\text{th}} \mapsto j^{\text{th}}-t (-i)^{\text{th}}$  row & {\tiny(both)}&$-j^{\text{th}} \mapsto -j^{\text{th}}+ti^{\text{th}}$  column\\
   \hline
ER3 &$-i^{\text{th}} \mapsto -i^{\text{th}}-tj^{\text{th}}$ row and & EC3 & $j^{\text{th}} \mapsto j^{\text{th}}+t(-i)^{\text{th}}$ column and \\
{\tiny(both)} &$-j^{\text{th}} \mapsto -j^{\text{th}}+ti^{\text{th}}$  row &{\tiny(both)}& $i^{\text{th}} \mapsto i^{\text{th}}-t(-j)^{\text{th}}$  column\\
  \hline
  ER4 & $0^{\text{th}} \mapsto 0^{\text{th}}-t(-i)^{\text{th}}$ row and & EC4 & $0^{\text{th}} \mapsto 0^{\text{th}}+2ti^{\text{th}}$ column and \\
{\tiny(odd)} &$i^{\text{th}} \mapsto i^{\text{th}}+2t0^{\text{th}}-t^{2}(-i)^{\text{th}}$  row &{\tiny(odd)}& $(-i)^{\text{th}} \mapsto (-i)^{\text{th}}-t0^{\text{th}}-t^{2}i^{\text{th}}$  column\\
 \hline
 ER5 &$0^{\text{th}} \mapsto 0^{\text{th}}+ti^{\text{th}}$ row and & EC5 & $0^{\text{th}} \mapsto 0^{\text{th}}-2t(-i)^{\text{th}}$ column and \\
{\tiny(odd)}& $(-i)^{\text{th}} \mapsto (-i)^{\text{th}}-2t0^{\text{th}}-t^{2}i^{\text{th}}$  row &{\tiny(odd)} &$i^{\text{th}} \mapsto i^{\text{th}}+t0^{\text{th}}-t^{2}(-i)^{\text{th}}$  column\\
 \hline
  ER6 & $0^{\text{th}} \mapsto 0^{\text{th}}+t(-i)^{\text{th}}$ row and & EC6 & $(-i)^{\text{th}} \mapsto (-i)^{\text{th}}+t0^{\text{th}}+t^{2}i^{\text{th}}$  column  \\
{\tiny(even)} &$i^{\text{th}} \mapsto 
i^{\text{th}}+t^{2}(-i)^{\text{th}}$  row &{\tiny(even)}&\\
 \hline
  ER7 & $0^{\text{th}} \mapsto 0^{\text{th}}+ti^{\text{th}}$ row and & EC7 & $i^{\text{th}} \mapsto i^{\text{th}}+t0^{\text{th}}+t^{2}(-i)^{\text{th}}$  column  \\
{\tiny(even)} &$(-i)^{\text{th}} \mapsto 
(-i)^{\text{th}}+t^{2}i^{\text{th}}$  row &{\tiny(even)}&\\
 \hline
 $ w_{i}$ & Interchange $i^{th}$ and $(-i)^{th}$ rows &$w_{i}$& Interchange $i^{th}$ and $(-i)^{th}$ column \\
 {\tiny(odd)} &with a sign change in $i^{th}, -i^{th} \;\text{and}\; 0^{th}$ rows&{\tiny(odd)} &with a sign change in $i^{th}, -i^{th} \;\text{and}\; 0^{th}$ columns\\
 \hline
 $ w_{i}$ & &$w_{i}$&  \\
 {\tiny(even)}&Interchange $i^{th}$ and $(-i)^{th}$ row & {\tiny(even)}&Interchange $i^{th}$ and $(-i)^{th}$ column\\
 \hline
 \end{tabular}
\end{table} 

\subsection{Gaussian elimination for $O(2l+1,k)$}
\begin{description}
\item [Step 1] Use ER1 and EC1 to make $A$ into a diagonal matrix but in the process it changes other matrices $A, B, C, D, E, F, X$, and $Y$. For the sake of notational convenience, we keep calling these changed matrices as $A, B, C, D, E, F, X$, and $Y$ as well. 
\item [Step 2] Now there will be two cases depending on the rank $\mathfrak{r}$ of matrix $A$. The rank of $A$ can be easily determined  using the number of non-zero diagonal entries. Use ER3 and non-zero diagonal entries of $A$ to make corresponding $\mathfrak{r}$ rows of $C$ zero.   
\begin{enumerate}
\item If $\mathfrak{r}=l$ then $C$ becomes zero matrix. 
\item If $\mathfrak{r}< l$ then interchange all zero rows of $A$ with corresponding rows of $C$ using $w_{i}$ so that the new $C$ becomes a zero matrix. 
\end{enumerate} 
Once $C$ becomes zero, note that Relation 5.2 if char$(k)$ is odd or Relation $Q(g(v))=Q(v)$ if char$(k)$ is even guarantees that $X$ becomes zero. Relation 5.5 guarantees that $A$ has full rank $l$ which also makes $D$ a diagonal with full rank $l$. Thus Relation 5.3 shows that $F$ becomes zero as well. Then use Step 1 to reduce $A=\begin{pmatrix}
1 &0&\cdots&0\\
0&1&\cdots&0\\
\vdots&\vdots&\ddots&\vdots\\
0&0&\cdots&\lambda
\end{pmatrix}$.
\item [Step 3] Now if char$(k)$ is even then Relation 5.4 guarantees that $E$ becomes zero as well.
 If char$(k)$ is odd then use ER4 to make $E$ a zero matrix.  
\item [Step 4] Use ER2 to make $B$ a zero matrix. For char$(k)$ even the relation $Q(g(v))=Q(v)$ guarantees that $Y$ is a zero matrix and for char$(k)$ odd Relation 5.4 implies that $Y$ becomes zero.
\end{description}
Thus the matrix $g$ reduces to 
$\begin{pmatrix}\pm 1&0&0\\0&A&0\\0&0&D\end{pmatrix}$, where $A=\text{diag}(1,\cdots,1,\lambda) $ and $D=\text{diag}(1,\cdots,1,\lambda^{-1})$.

\begin{proof}[Proof of Corollary C]  Let $k$ be a perfect filed of characteristics 2. Note that we can write the diagonal matrix $\rm{diag}(1,\cdots, 1,\lambda, 1,\cdots, 1,\lambda^{-1} )$ as a product of elementary matrices as follows: 
\begin{align*}
\rm{diag}(1,\cdots, 1,\lambda, 1,\cdots, 1,\lambda^{-1} )&=x_{l,-l}(t)x_{-l,l}(-t^{-1})x_{l,-l}(t), \; \rm{where}\; t^2=\lambda.
\end{align*} and hence we can reduce the matrix $g$ to identity.  
 \end{proof}

\subsection{Elementary matrices (Chevalley generators) for twisted orthogonal groups $O^{-}(2l, k)$} 
In this section, we describe row-column operations for twisted Chevalley groups. These groups are also known as the Steinberg groups. An element $g\in O^{-}(2l,k)$ is denoted $g=\begin{pmatrix}A_{0}&X&Y\\E&A&B\\F&C&D\end{pmatrix}$, where $A,B,C$ \& $D$ are $(l-1) \times (l-1)$ matrices, $X$ and $Y$ are $2 \times (l-1)$ matrices, $E$ and $F$ are $(l-1) \times 2$ matrices and $A_{0}$ is a $2\times 2$ matrix. In the Gaussian elimination algorithm that we discuss, we reduce $X,Y,E \& F$ to zero and $A$ and $D$ to diagonal matrices. However, unlike the previous cases we were unable to reduce $A_0$ to an identity matrix. However, for odd characteristics we were able to reduce $A_0$ to a two-parameter subgroup. Furthermore, our algorithm provides for the spinor norm as before.  

Let us go ahead and talk about the output of the algorithm. In the output we will have a $2\times 2$ block (also called $A_0$) which will satisfy $\tr A_0\beta_0 A_0=\beta_0$, where $\beta_0=\begin{pmatrix} 1&0\\0&\epsilon\end{pmatrix}$ for odd characteristics, as defined earlier. Then $A_0$ is a orthogonal group given by the bilinear form $\beta_0$. Now if we write $A_0=\begin{pmatrix} a & b\\c &d\end{pmatrix}$, then we get the following equation:
\begin{eqnarray*}
a^2+c^2\epsilon=1\\
ab+cd\epsilon=1\\
b^2+d^2\epsilon=\epsilon
\end{eqnarray*} 
Considering the fact that det$(A_0)=\pm 1$, one more equation $ac-bd=\pm 1$ and this leads to two cases either $a=d$ and $b=c\epsilon$ or $a=-d$ and $b=-c\epsilon$. Recall that, since $\epsilon$ is not a square, $d\neq 0$. Then if $c=0$, then there are four choices for $A_0$ and these are $A_0=\begin{pmatrix}
\pm 1& 0\\ 0  & \pm 1
\end{pmatrix} 
$.

To summarize, in the output of the algorithm $A_0$ will have either of the six forms
\begin{eqnarray}
\begin{pmatrix}
t & -s\epsilon\\
s & t
\end{pmatrix}
&
\text{or}\;\;
\begin{pmatrix}
t & s\epsilon\\
s & -t
\end{pmatrix} & \text{where}\;\; t^2+s^2\epsilon=1\\
\text{or} & \begin{pmatrix}
\pm 1 & 0\\
0 & \pm 1
\end{pmatrix}.
\end{eqnarray}
There are now two ways to describe the algorithm, one is to leave $A_0$ as it is in the output of the algorithm and the other is to include these matrices as generators. For the purpose of uniform exposition we choose the later and included the following two generators
\begin{eqnarray*}
&x_{1}(t, s)=I+(t-1)e_{1,1}-(t+1) e_{-1,-1}+s(e_{-1,1}+\epsilon e_{1,-1}); \;\; t^2+\epsilon s^2=1 \\
&x_{2}=I-2e_{-1, -1}
\end{eqnarray*}
in the list of elementary generators in Table~\ref{table:twisted}. In the case of even characteristics no such reduction is possible and we included the matrix $\begin{pmatrix}
t & p\\
r & s
\end{pmatrix}
$ in the list of generators with the condition that the determinant is 1. 

The elementary matrices for $O^{-}(2l, k)$ depends on characteristics of $k$. We describe them separately in the following table. Let $\alpha$ be an Arf-invariant, $2\leq i,j\leq l$ and $t\in K$, $\xi \in k^{*}$.
\begin{table}[ht]\label{elementary_gens_twist}
\caption{Elementary matrices for $O^{-}(2l, k)$} 
\centering 
\begin{tabular}{| c | c | c c |} 
\hline 
& & &\\
  char$(k)$ & & Elementary matrices& \\ [0.5ex]
 \hline 
 &  $x_{i,j}(t)$&$I+t(e_{i,j}- e_{-j,-i})$, & $i\neq j$\\
both  & $x_{i,-j}(t)$&$I+t(e_{i,-j}- e_{j,-i})$,& $i<j$\\
  & $x_{-i,j}(t)$&$I+t(e_{-i,j}-e_{-j,i})$, & $i<j$\\
  & $w_{i}$&$I-e_{i,i}-e_{-i,-i}+e_{i,-i}+e_{-i,i}$& $2\leq i\leq l$\\
  \hline
 &$x_{i,1}(t)$&$I+t(e_{1,i}-2e_{-i,1})-t^{2}e_{-i,i}$& $2\leq i\leq l$\\
& $x_{1,i}(t)$&$I+t(-e_{1,-i}+2e_{i,1})-t^{2}e_{i,-i}$& $2\leq i\leq l$\\
&$x_{i,-1}(t)$&$I+t(e_{-1,i}-2\epsilon e_{-i,-1})-\epsilon t^{2}e_{-i,i}$& $2\leq i\leq l$\\
odd&$x_{-1,i}(t)$&$I+t(-e_{-1,-i}+2\epsilon e_{i,-1})-\epsilon t^{2}e_{i,-i}$& $2\leq i\leq l$\\
&$x_{1}(t, s)$&$I+(t-1)e_{1,1}-(t+1) e_{-1,-1}+s(e_{-1,1}+\epsilon e_{1,-1})$, & $t^2+\epsilon s^2=1$ \\
&$x_{2}$&$I-2e_{-1, -1}$&\\
\hline
&$x_{1,-i}(t)$&$I+te_{1,-i}+te_{i,-1}+\alpha t^{2}e_{i,-i}$& $2\leq i\leq l$\\
even&$x_{-1,-i}(t)$&$I+te_{-1,-i}+te_{i,1}+\alpha t^{2}e_{i,-i}$& $2\leq i\leq l$\\
&$x_{A_{0}}$&$I+(t-1)e_{1,1}+(s-1)e_{-1,-1}+pe_{1,-1}+re_{-1,1}$, & $ts+pr=1$.\\
\hline
\end{tabular}
\label{table:twisted} 
\end{table}

 Let us note the effect of multiplying $g$ by elementary matrices. 
 Elementary matrices for the twisted orthogonal group in even characteristics differs from that of odd characteristics so in the following table (Table~\ref{Table3} \& ~\ref{Table4}) we made that distinction and listed them separately in different row according to the characteristics of $k$.  
 \begin{table}[h!]
\caption{The row operations for $O^{-}(2l, k)$} \label{Table3}
\vskip 0.2 cm
\centering 
\scalebox{0.95}{
\begin{tabular}{|c c|} 
\hline 
&\\
 &\textbf{Row operations}   \\ [0.5ex]
 \hline 
ER1{\tiny (both)}&$i^{\text{th}} \mapsto i^{\text{th}}+t j^{\text{th}}$ row and and 
 $-j^{\text{th}} \mapsto -j^{\text{th}}-t(-i)^{\text{th}}$  row \\
 \hline
ER2 {\tiny (both)} &$i^{\text{th}} \mapsto i^{\text{th}}+t (-j)^{\text{th}}$ row and $j^{\text{th}} \mapsto j^{\text{th}}-t (-i)^{\text{th}}$  row  \\
  \hline
ER3{\tiny (both)} & $-i^{\text{th}} \mapsto -i^{\text{th}}-tj^{\text{th}}$ row and  $-j^{\text{th}} \mapsto -j^{\text{th}}+ti^{\text{th}}$  row \\
  \hline
ER4{\tiny (odd)} & $1^{\text{st}} \mapsto 1^{\text{st}}-t(-i)^{\text{th}}$ row and  $i^{\text{th}} \mapsto i^{\text{th}}+2t1^{\text{st}}-t^{2}(-i)^{\text{th}}$  row \\
 \hline
ER5{\tiny (odd)} & $1^{\text{st}} \mapsto 1^{\text{st}}+ti^{\text{th}}$ row and  $(-i)^{\text{th}} \mapsto (-i)^{\text{th}}-2t1^{\text{st}}-t^{2}i^{\text{th}}$  row \\
 \hline
ER6{\tiny (odd)} &  $(-1)^{\text{th}} \mapsto (-1)^{\text{th}}-t(-i)^{\text{th}}$ row and 
  $i^{\text{th}} \mapsto i^{\text{th}}+2\epsilon t(-1)^{\text{th}}-\epsilon t^{2}(-i)^{\text{th}}$  row \\
 \hline
ER7{\tiny (odd)} & $(-1)^{\text{th}} \mapsto (-1)^{\text{th}}+ti^{\text{th}}$ row and 
 $(-i)^{\text{th}} \mapsto (-i)^{\text{th}}-2\epsilon t(-1)^{\text{th}}-\epsilon t^{2}i^{\text{th}}$  row\\
\hline 
ER8 {\tiny (even)} & $1^{\text{st}} \mapsto 1^{\text{st}}+t(-i)^{\text{th}}$ row and $i^{\text{th}} \mapsto i^{\text{th}}+t(-1)^{\text{th}}+\alpha t^{2}(-i)^{\text{th}}$  row \\
 \hline
 ER9{\tiny (even)}  & $(-1)^{\text{th}} \mapsto (-1)^{\text{th}}+t(-i)^{\text{th}}$ row and $i^{\text{th}} \mapsto i^{\text{th}}+t1^{\text{st}}+\alpha t^{2}(-i)^{\text{th}}$  row \\
 \hline
$ w_{i}${\tiny(both)} & Interchange $i^{th}$ and $(-i)^{th}$ row \\
 \hline
 \end{tabular}}
\end{table}

\begin{table}[h!]
\caption{The column operations for $O^{-}(2l, k)$} \label{Table4}
\vskip 0.2 cm
\centering 
\scalebox{0.95}{
\begin{tabular}{|c c|} 
\hline 
&\\
&\textbf{Column operations}  \\ [0.5ex]
 \hline 
EC1{\tiny(both)} &$j^{\text{th}} \mapsto j^{\text{th}}+t i^{\text{th}}$ column and 
 $-i^{\text{th}} \mapsto -i^{\text{th}}-t(-j)^{\text{th}}$  column\\
 \hline
EC2{\tiny (both)} & $-i^{\text{th}} \mapsto -i^{\text{th}}-tj^{\text{th}}$ column and 
 $-j^{\text{th}} \mapsto -j^{\text{th}}+ti^{\text{th}}$  column\\
   \hline
 EC3{\tiny (both)} &$j^{\text{th}} \mapsto j^{\text{th}}+t(-i)^{\text{th}}$ column and 
  $i^{\text{th}} \mapsto i^{\text{th}}-t(-j)^{\text{th}}$  column\\
  \hline
 EC4{\tiny (odd)} & $1^{\text{st}} \mapsto 1^{\text{st}}+2ti^{\text{th}}$ column and 
 $(-i)^{\text{th}} \mapsto (-i)^{\text{th}}-t1^{\text{st}}-t^{2}i^{\text{th}}$  column\\
 \hline
  EC5{\tiny (odd)} &$1^{\text{st}} \mapsto 1^{\text{st}}-2t(-i)^{\text{th}}$ column and 
  $i^{\text{th}} \mapsto i^{\text{th}}+t1^{\text{st}}-t^{2}(-i)^{\text{th}}$  column\\
 \hline
  EC6  {\tiny (odd)}  &$(-1)^{\text{th}} \mapsto (-1)^{\text{th}}+(2\epsilon t)i^{\text{th}}$ column and 
$(-i)^{\text{th}} \mapsto (-i)^{\text{th}}-t(-1)^{\text{th}}-\epsilon t^{2}i^{\text{th}}$  column\\
 \hline
  EC7 {\tiny (odd)} & $(-1)^{\text{th}} \mapsto (-1)^{\text{th}}-2\epsilon t(-i)^{\text{th}}$ column and 
  $i^{\text{th}} \mapsto i^{\text{th}}+t(-1)^{\text{th}}-\epsilon t^{2}(-i)^{\text{th}}$  column\\
 \hline 
 EC8 {\tiny (even)} &$(-1)^{\text{th}} \mapsto (-1)^{\text{th}}+ti^{\text{th}}$ column and 
  $(-i)^{\text{th}} \mapsto (-i)^{\text{th}}+t1^{\text{st}}+\alpha t^{2}i^{\text{th}}$  column\\
 \hline
 EC9{\tiny (even)} & $1^{\text{st}} \mapsto 1^{\text{st}}+ti^{\text{th}}$ column and 
    $(-i)^{\text{th}} \mapsto (-i)^{\text{th}}+t(-1)^{\text{th}}+\alpha t^{2}i^{\text{th}}$  column\\
 \hline
$ w_{i}${\tiny (both)} & Interchange $i^{th}$ and $(-i)^{th}$ column \\
 \hline
 \end{tabular}}
\end{table}

Note that any isometry $g$ satisfying the quadratic form $Q$ also  satisfy $\tra{g}\beta g=\beta $.   The main reason the following algorithm works is the closed condition 
$\tra{g}\beta g=\beta $ which gives the  following relations:
\begin{align}
\tra{A_{0}}\beta_{0} A_{0}+\tra{F}E+\tra{E}F&=\beta_{0},\\
\tra{A_{0}}\beta_{0}X+\tra{F}A+\tra{E}C&=0, \\ 
\tra{A_{0}}\beta_{0}Y+\tra{F}B+\tra{E}D&=0,\\
\tra{X}\beta_{0}X+\tra{C}A+\tra{A}C&= 0, \\ \tra{X}\beta_{0}Y+\tra{C}B+\tra{A}D&=I_{l-1}.
\end{align}
\subsection{The Gaussian elimination algorithm for $O^{-}(2l, k)$}
\begin{description}
\item [Step 1] Use ER1 and EC1 to make $A$ into a diagonal matrix but in the process it changes other matrices $A_{0}$, $A, B, C, D, E, F, X$, and $Y$. For sake of notational convenience, we keep calling these changed matrices as $A_{0}$, $A, B, C, D, E, F, X$, and $Y$ as well. 
\item [Step 2] Now there will be two cases depending on the rank $\mathfrak{r}$ of the matrix $A$. The rank of $A$ can be easily determined by the number of non-zero diagonal entries. 
\item [Step 3] Use ER3 and non-zero diagonal entries of $A$ to make corresponding $\mathfrak{r}$ rows of $C$ zero.   
\begin{enumerate}
\item[] If $\mathfrak{r}=l-1$ then $C$ becomes zero matrix. 
\item[] If $\mathfrak{r}<l-1$ then interchange all zero rows of $A$ with corresponding rows of $C$ using $w_{i}$ so that the new $C$ becomes a zero matrix. 
\end{enumerate} 
Once $C$ becomes zero one can note that the relation   
$\tra{X}\beta_{0}X+\tra{C}A+\tra{A}C= 0$ if char$(k)$ is odd or the relation $Q(g(v))=Q(v)$ and the fact that $\alpha t^{2}+t+\alpha$ is irreducible when char$(k)$ is even guarantees that $X$ becomes zero (see Lemma 5.7). Then the relation $\tra{A_{0}}\beta_{0}X+\tra{F}A+\tra{E}C=0$ shows that $F$ becomes zero as well and the relation  $\tra{X}\beta_{0}Y+\tra{C}B+\tra{A}D= I_{l-1}$ guarantees that $A$ has full rank $l-1$ which also makes $D$ a diagonal with full rank $l-1$. Now we diagonalize $A$ again to the form 
$\begin{pmatrix}
1 &0&\cdots&0\\
0&1&\cdots&0\\
\vdots&\vdots&\ddots&\vdots\\
0&0&\cdots&\lambda
\end{pmatrix}$ as in Step 1.
\item [Step 4] Use EC4 and EC6 when char$(k)$ is odd or use EC8 and EC9 when char$(k)$ is even to make $E$ zero.  

Note that the relation $\tra{A_{0}}\beta_{0} A_{0}+\tra{F}E+\tra{E}F=\beta_{0}$ shows that $A_{0}$ is invertible. Thus the relation $\tra{A_{0}}\beta_{0}Y+\tra{F}B+\tra{E}D=0$ guarantees that $Y$ becomes zero.  
\item [Step 5] Use ER2 to make $B$ a zero matrix. Thus the matrix $g$ reduces to:
$g=\begin{pmatrix}A_{0}&0&0\\0&A&0\\0&0&D\end{pmatrix}$, where $A=\diag(1,1,\cdots,1,\lambda) $ and $D=\text{diag}(1,1,\cdots,1,\lambda^{-1})$. Now if char$(k)$ is odd then go to step 6 otherwise go to step 7.
 \item [Step 6] Using the relation $\tra{A_{0}}\beta_{0} A_{0}=\beta_{0}$ it is easy to check that $A_{0}$ has the form $\begin{pmatrix}t&-\epsilon s\\ s& t\end{pmatrix}$ or $\begin{pmatrix}t&\epsilon s\\ s& -t\end{pmatrix}$. If the determinant of $A_{0}$ is $-1$, multiply $g$ by $x_{2}$ to get new $g$ of the above form such that $A_{0}$ has determinant 1. Now using the elementary matrix $x_{1}(t,s)$ we can reduce $g$ to $\begin{pmatrix}I_{2}&0&0\\0&A&0\\0&0&D\end{pmatrix}$, where $A$ and $D$ are as above. 
\item [Step 7] Using elementary matrix $x_{A_{0}}$  we can reduce  $g$ to $\begin{pmatrix}I_{2}&0&0\\0&A&0\\0&0&D\end{pmatrix}$, where $A=\diag(1,1,\cdots,1,\lambda) $ and $D=\text{diag}(1,1,\cdots,1,\lambda^{-1})$. 
\end{description}
\begin{lemma}
Let $k$ be a field of characteristics 2 and let $g=\begin{pmatrix}A_{0}&X&Y\\E&A&B\\F&0&D\end{pmatrix}$, where \,$A=\diag(1,1,\cdots,1,\lambda)$, be an element of $O^{-}(2l, k)$ then $X=0$.
\end{lemma}
\begin{proof}
Let $\{e_{1}, e_{-1}, e_{2}, \cdots, e_{l}, e_{-2},\cdots, e_{-l} \}$ be the standard basis of the vector space $V$. Recall that  for a column vector $x=(x_{1}, x_{-1}, x_{2}, \cdots, x_{l}, x_{-2}, \cdots, x_{-l})^{t}$ the action of the quadratic form $Q$ is given by $Q(x)=\alpha(x_{1}^{2}+x_{-1}^{2}) +x_{1}x_{-1}+\ldots +x_{l}x_{-l}$, where $\alpha t^2 +t+\alpha$ is irreducible over $k[t]$. By definition, for any $g\in O^{-}(2l, k)$ we have $Q(g(x))=Q(x)$ for all $x\in V$. Let $X=\begin{pmatrix} x_{11}\cdots x_{1(l-1)}\\x_{21}\cdots x_{2(l-1)}\end{pmatrix}$ be a $2\times (l-1)$ matrix. Computing $Q(g(e_{i}))=Q(e_{i})$ for all $2\leq i \leq l$, we can see that  $\alpha(x_{1i}^2+x_{2i}^2)+x_{1i}x_{2i}=0$. If $x_{2i}=0$ then we can see that $x_{1i}=0$. Suppose $x_{2i}\neq 0$ for some $i$ then we rewrite the equation by dividing it by $x_{2i}$ as $\alpha(\frac{x_{1i}}{x_{2i}})^2+\frac{x_{1i}}{x_{2i}}+\alpha=0$, which is a contradiction to the fact that $\alpha t^2 +t+\alpha$ is irreducible over $k[t]$. Thus, $x_{2i}=0$ for all $2\leq i\leq l$ and hence $X=0$.
\end{proof}

\subsection{Time-complexity of the above algorithm} We establish that
the worst case time-complexity of the above algorithm is $\mathrm{O}(l^3)$. We mostly count the number of field multiplications.
\begin{itemize}
\item[Step 1] We make $A$ a diagonal matrix by row-column operations. That has complexity $\mathrm{O}(l^3)$.
\item[Step 2] In making both $C$ and $B$ zero-matrix we  multiply two rows by a field
  element and additions. In the worst case, it has to be done $O(l)$
  times and done $O(l^2)$ many times. So the complexity is $\mathrm{O}(l^3)$.
\item[Step 3] In odd-orthogonal group and twisted orthogonal group we clear $X,Y,E,F$, this clearly has complexity $O(l^2)$
\item[] Step 4 has only a few steps that is independent of $l$.
 \end{itemize}
Then clearly, the time-complexity of our algorithm is $\mathrm{O}(l^3)$.
\section{Computing spinor norm for orthogonal groups}\label{spinornorm}

In this section, we show how we can use our Gaussian elimination algorithm to compute the spinor norm for orthogonal groups. Throughout this section we assume that the field $k$ is of odd characteristics. The classical way to define spinor norm is via Clifford algebras~\cite[Chapters 8 \& 9]{gr}. Spinor norm is a group homomorphism $\Theta\colon \text{O}(d,k)\rightarrow k^{\times}/{k^{\times}}^2$, restriction of which to $\text{SO}(d,k)$ gives $\Omega(d,k)$ as kernel. 
However, in practice, it is difficult to use that definition to compute the spinor norm. Wall~\cite{wa}, Zassenhaus~\cite{za} and Hahn~\cite{ha} developed a theory to compute the spinor norm. For our exposition, we follow~\cite[Chapter 11]{ta}. 

Let $g$ be an element of the orthogonal group. Let $\tilde g=I-g$ and $V_g=\tilde g(V)$ and $V^g=ker(\tilde g)$. Using $\beta$ we define Wall's bilinear form $[\ ,\ ]_g$ on $V_g$ as follows:
$$[u,v]_g=\beta(u,y),\ \text{where},\ v=\tilde g(y).$$
This bilinear form satisfies following properties:
\begin{enumerate}
\item $[u,v]_g+[v,u]_g=\beta(u,v)$ and $[u,u]_g=Q(u)$ for all $u,v\in V_g$.
\item $g$ is an isometry on $V_g$ with respect to $[\ ,\ ]_g$.
\item $[v,u]_g=-[u,gv]$ for all $u,v\in V_g$.
\item $[\ ,\ ]_g$ is non-degenerate. 
\end{enumerate}
Then the \textbf{spinor norm} is 
$$\Theta(g)= \overline{\text{disc}(V_g,[\ ,\ ]_g)}\;  \text{if}\; g\neq I$$ 
extended to $I$ by defining $\Theta(I)=\overline 1$.
An element $g$ is called regular if $V_g$ is non-degenerate subspace of $V$ with respect to the form $\beta$. 
Hahn~\cite[Proposition 2.1]{ha} proved that for a regular element $g$ the spinor norm is $\Theta(g)=\overline{\det(\tilde g|_{V_g})\text{disc}(V_g)}$. This gives,
\begin{proposition}\label{propospinor}
\begin{enumerate}
\item For a reflection $\rho_v$, $\Theta(\rho_v)=\overline{Q(v)}$.
\item $\Theta(-1)=\overline{\text{disc}(V,\beta)}$.
\item For a unipotent element $g$ the spinor norm is trivial, i.e., $\Theta(g)=\overline{1}$.
\end{enumerate}
\end{proposition}
\noindent Murray and Roney-Dougal~\cite{mr} used the formula of Hahn to compute spinor norm. 
However, we show (Corollary B) that the Gaussian elimination algorithm we developed in Section~\ref{wordproblem} ouptuts the spinor norm.
 We compute the spinor norm for elements in twisted orthogonal group separately. First we observe the following:
\begin{lemma}\label{spinornorm1}
For the group $\text{O}^{+}(d,k)$, $d\geq 4$,
 \begin{enumerate}
  \item $\Theta(x_{i,j}(t))=\Theta(x_{-i,j}(t))=\Theta(x_{i,-j}(t))=\overline{1}$. Furthermore, in odd case we also have $\Theta(x_{i,0}(t))=\overline 1=\Theta(x_{0,i}(t))$.
  \item $\Theta(w_l)=\overline{1}$.
  \item $\Theta(\diag(1,\ldots,1,\lambda,1,\ldots,1,\lambda^{-1}))=\overline{\lambda}$.
 \end{enumerate}
\end{lemma}
\begin{proof}
We use Proposition~\ref{propospinor}. The first claim follows from the
fact that all elementary matrices are unipotent. 
 The element $w_l=\rho_{(e_l+e_{-l})}$ is a reflection thus $\Theta(w_l)=\overline{Q(e_l+e_{-l})}=\overline{1}$.
 
For the third part we note that $\diag(1,\ldots,1,\lambda,1,\ldots,1,\lambda^{-1}) = \rho_{(e_l+e_{-l})} \rho_{(e_l+\lambda e_{-l})}$ and hence the spinor norm $\Theta(\diag(1,\ldots,1,\lambda,1,\ldots,1,\lambda^{-1}) )=\Theta(\rho_{(e_l+\lambda e_{-l})}) =\overline{Q(e_l+\lambda e_{-l})}=\overline{\lambda}$.
\end{proof}

\begin{lemma} The spinor norm of elementary matrices in $O^{-}(d, K)$ are:
 \begin{enumerate}
  \item $\Theta(x_{i,j}(t)) = \Theta(x_{-i, j}(t)) = \Theta(x_{i,-j}(t)) = \bar{1}.$
  \item $\Theta(w_{i}) = \bar{1}.$
  \item $ \Theta(\text{diag}(1,1,1,...,1,\lambda,1,...,1,\lambda^{-1})) = \bar{\lambda},\\
  \Theta(\text{diag}(1,-1,1,...,1,\lambda,1,...,1,\lambda^{-1})) = \overline{\epsilon\lambda},\\
  \Theta(\text{diag}(-1,1,1,...,1,\lambda,1,...,1,\lambda^{-1})) = \bar{\lambda},\\
  \Theta(\text{diag}(-1,-1,1,...,1,\lambda,1,...,1,\lambda^{-1})) = \overline{\epsilon\lambda}$.
  \item $ \Theta(x_{1}(t,s))=\overline{2(1-t)}$ whenever $t\neq 1$.
  \item $\Theta(x_{2})=\overline{\epsilon}$
 \end{enumerate}
 \end{lemma}
 Proof: The first one follows from previous proposition as all elementary matrices are unipotent. 
 The element $w_{i} = \rho_{e_{i}+e_{-i}}$ is a reflection thus 
 $\Theta(w_{i})= \overline{Q(e_{i} +e_{-i})} = \bar{1}$.\\ For the third part we note that
 $\text{diag}(1,1,1,...,1,\lambda,1,...,1,\lambda^{-1}) = \rho_{e_{l}+e_{-l}}\rho_{e_{l}+\lambda e_{-l}}$
and hence,
\begin{align*}
\Theta(\text{diag}(1,1,1,...,1,\lambda,1,...,1,\lambda^{-1}))& = \Theta(\rho_{e_{l}+e_{-l}})\Theta(\rho_{e_{l}+\lambda e_{-l}}) \\
\Theta(\text{diag}(1,1,1,...,1,\lambda,1,...,1,\lambda^{-1}))&=\overline{Q(e_{l}+\lambda e_{-l})} \\
\Theta(\text{diag}(1,1,1,...,1,\lambda,1,...,1,\lambda^{-1}))&= \bar{\lambda}.
\end{align*}
Observe that \linespread{0.5}
\begin{align*}
\text{diag}(1,-1,1,...,1,\lambda,1,...,1,\lambda^{-1})&=\text{diag}(1,-1,1,...,1,\lambda,1,...,1,\lambda^{-1})\rho_{e_{-1}} \\
 \text{diag}(-1,1,1,...,1,\lambda,1,...,1,\lambda^{-1})&=\text{diag}(1,-1,1,...,1,\lambda,1,...,1,\lambda^{-1})\rho_{e_{1}}
\end{align*}
implies that 
\begin{align*}
\Theta(\text{diag}(1,-1,1,...,1,\lambda,1,...,1,\lambda^{-1}))&=\overline{\epsilon\lambda},\\
\Theta(\text{diag}(-1,1,1,...,1,\lambda,1,...,1,\lambda^{-1}))&=\overline{\lambda}.
\end{align*} 
Similarly we can see that
\begin{align*}
\text{diag}(-1,-1,1,...,1,\lambda,1,...,1,\lambda^{-1}))&=\text{diag}(1,-1,1,...,1,\lambda,1,...,1,\lambda^{-1})\rho_{e_{1}}
\end{align*} 
implies that  $\Theta(\text{diag}(-1,-1,1,...,1,\lambda,1,...,1,\lambda^{-1})) = \overline{\epsilon\lambda}$.\\
For the fourth part we observe that 
$x_{1}(t,s)=\rho_{(t-1)e_{1}+se_{-1}}$ and hence
\begin{align*}
\Theta(x_{1}(t, s))&=\overline{Q(\rho_{(t-1)e_{1}+se_{-1}})}\\
&=(t-1)^2+\epsilon s^2 \\
&=2(1-t)  \quad \quad \text{as} \quad t^2+\epsilon s^2=1. 
\end{align*}
Note that $x_{2}=\rho_{e_{-1}}$ and $\Theta(\rho_{e_{-1}})=\bar{\epsilon} $ implies that $\Theta(x_{2})=\overline{\epsilon}$. 
\begin{proof}[Proof of Corollary B]
 Let $g\in\text{O}^+(d,k)$. From Theorem A, we write $g$ as a product of
 elementary matrices and a diagonal matrix $\diag(1,\ldots,1,\lambda,1,\ldots,1,\lambda^{-1})$ and hence we can find the spinor norm of $g$.
\end{proof}
\subsection{Computing the spinor norm in $O^{-}(d,k)$}\label{section:twisted_sn}
Before moving to Step 6 of the algorithm, look at $A_0$. If det$(A_0)=1$, it is of the form 
$\begin{pmatrix}
t & -s\epsilon\\
s & t
\end{pmatrix} 
$
where $s\neq 0$. Furthermore, if $s=0$ then it is of the form
$\begin{pmatrix}
1&0\\
0&1
\end{pmatrix}
$ or $\begin{pmatrix}
-1&0\\
0&-1
\end{pmatrix}$ the spinor norm then can be computed from the lemma above. If the determinant is $-1$, then $A_0$ is $x_2$ times one of the previous matrices. So the spinor norm can be computed because it is multiplicative.
\section{Double coset decomposition for Siegel maximal parabolic}\label{parabolicdecomp}

In this section, we compute the double coset decomposition with respect to Siegel maximal parabolic subgroup using our algorithm. Let $P$ be the Siegel maximal parabolic of $G$ where $G$ is either split orthogonal group $\text{O}(d,k)$ or $\text{Sp}(2l,k)$, where char(k) is odd. In Lie theory, a parabolic is obtained by fixing a subset of simple roots~\cite[Section 8.3]{ca}. Siegel maximal parabolic corresponds to the subset consisting of all but the last simple root. 
Geometrically, a parabolic subgroup is obtained as fixed subgroup of a totally isotropic flag~\cite[Proposition 12.13]{mt}. The Siegel maximal parabolic is the fixed subgroup of  following isotropic flag (with the basis in Section~\ref{des2}): $$\{0\}\subset \{e_1,\ldots,e_l\}\subset V.$$
Thus $P$ is of the form
         $\begin{pmatrix} \alpha&0&Y\\ E&A&B\\F&0&D\end{pmatrix}$ in  $\text{O}(2l+1,k)$ 
         and $\begin{pmatrix} A&B\\0&D\end{pmatrix}$ in
         $\text{Sp}(2l,k)$ and $\text{O}(2l,k)$.

The problem is to get the double coset decomposition $P\backslash G/ P$. That is, we want to write $G=\underset{\omega\in\widehat{W}}{\bigsqcup} P\omega P $ as disjoint union where $\widehat{W}$ is
a finite subset of $G$. Equivalently, given $g\in G$ we need an
algorithm to determine the unique $\omega\in \widehat{W}$ such that $g\in
P\omega P$. 
If $G$ is connected with Weyl group $W$ and suppose $W_P$ is the Weyl group corresponding to $P$ then~\cite[Proposition 2.8.1]{ca2}  
$$ P\backslash G/P \longleftrightarrow W_P\backslash W/W_P.$$
We need a slight variation of this as the orthogonal group is not connected. We define $\widehat{W}$ as follows:

\[\widehat{W}=\left\{\omega_0=I, \omega_i=w_{1}\cdots w_{i}\mid 1\leq i\leq l\right\}\]
where $w_i$ were defined earlier for each class of groups.

\begin{theorem}\label{thmbruhatdecomposition}
Let $P$ be the Siegel maximal parabolic subgroup in $G$, where $G$ is either $\text{O}(d,k)$ or
$\text{Sp}(d,k)$.  Let $g\in G$. Then there is an efficient algorithm
to determine $\omega$ such that $g\in P\omega P$. Furthermore, $\widehat
W$ the set of all $\omega$ is a finite set of $l+1$ elements where $d=2l$ or $2l+1$.
\end{theorem}
\begin{proof}
 In this proof we proceed with a similar but slightly different Gaussian
 elimination algorithm. Recall that
 $g=\begin{pmatrix}A&B\\C&D\end{pmatrix}$ whenever $g$ belongs to
 $\text{Sp}(2l,k)$ or $\text{O}(2l,k)$ or
 $g=\begin{pmatrix}\alpha&X&Y\\ E&A&B\\F&C&D\end{pmatrix}$ whenever
 $g$ belongs to $\text{O}(2l+1,k)$. In our algorithm, we made $A$ into a diagonal
 matrix. Instead of that, we can use elementary matrices ER1 and EC1 to
 make $C$ into a diagonal matrix and then do the row interchange to
 make $A$ into a diagonal matrix and $C$ a zero matrix. If we do that,
 we note that elementary matrices E1 and E2 are in
 $P$. The proof is just keeping track of elements of
 $P$ in this Gaussian elimination algorithm. The step 1 in the
 algorithm says that there are elements $p_1, p_2\in P$
 such that $p_1gp_2 = \begin{pmatrix} A_1&B_1\\C_1&D_1\end{pmatrix}$
 where $C_1$ is a diagonal matrix with $m$ non-zero entries. Clearly
 $m=0$ if and only if $g\in P$. In that case $g$ is in the double
 coset $P\omega_0P=P$. Now suppose $m\geq 1$. Then in Step 2 we
 multiply by E2 to make the first $m$ rows of $A_1$ zero, i.e., there is a $p_3\in P$ such that $p_3p_1gp_2 = \begin{pmatrix} \tilde A_1 & \tilde B_1\\C_1&D_1\end{pmatrix}$ where first $m$ rows of $\tilde A_1$ are zero. After this we interchange rows $i$ with $-i$ for $1\leq i\leq m$ which makes $C_1$ zero, i.e., multiplying by $\omega_m$ we get  $\omega_m p_3p_1gp_2=\begin{pmatrix} A_2&B_2\\0&D_2\end{pmatrix}\in P$. Thus $g\in P\omega_m P$.
 
 For $\text{O}(2l+1,k)$ we note that the elementary matrices E1, E2
 and E4a are in $P$. Rest of the proof is similar to the earlier case
 and follows by carefully keeping track of elementary matrices used in
 our algorithm in Section~\ref{wordproblem}. 
\end{proof}
\section{Conclusions}

Now some implementation results, we implemented our algorithm in magma~\cite{magma}. We found
our implementation to be fast and stable. In magma, Costi and C.~Schneider
installed a function \emph{ClassicalRewriteNatural}. It is the
row-column operation developed by Costi~\cite{costi} in natural
representation. We tested the time
taken by our algorithm and the one taken by the Magma function. To do
this test, we followed Costi~\cite[Table 6.1]{costi} as closely as
possible. Two kind of simulations were done. In one case, we
fixed the size of the field at $7^{10}$ and varied the size of the
matrix from $20$ to $60$. To time both these algorithms for any
particular input, we took one thousand random samples from the group and run the
algorithm for each one of them. Then the final time was the average of
this one thousand random repetitions. The times were tabulated and
presented below.

In the other case, we kept the size of the matrix fixed at $20$ and we
varied the size of the field, keeping the characteristic fixed at $7$. In many cases the magma computation for
the function ClassicalRewriteNatural will not stop in a reasonable amount
of time or will give an error and not finish computing. In those cases,
though our algorithm worked perfectly, we were unable to get adequate
data to plot and are represented by gaps in the graph drawn. Here also the times are the
average of one thousand random repetitions.  
\begin{figure}[h]
\centering
\includegraphics[scale=0.40]{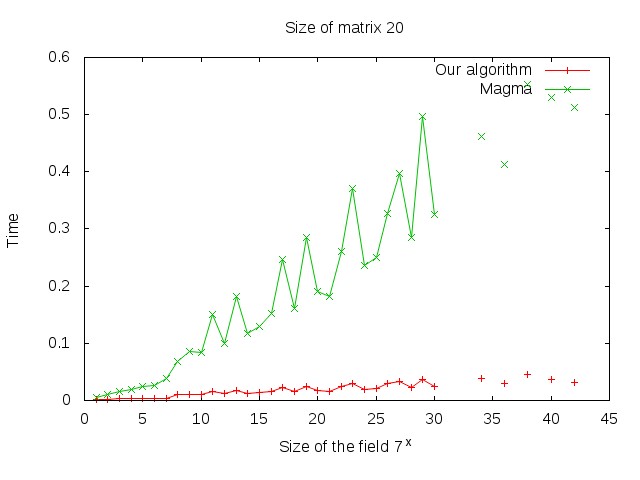}
\includegraphics[scale=0.40]{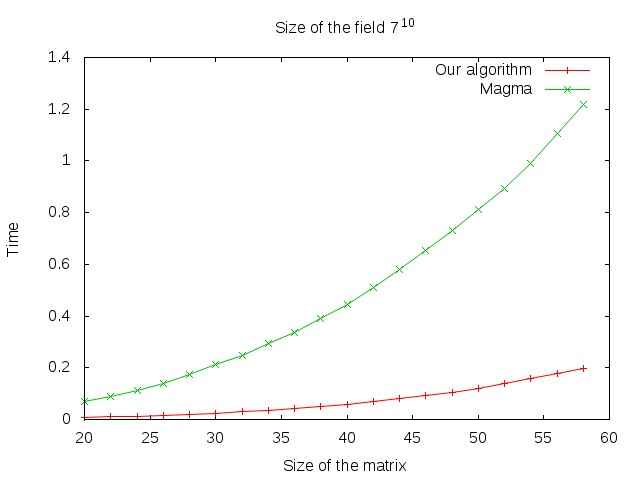}
\caption{Some simulations comparing our algorithm with the one inbuilt in Magma for even-order orthogonal groups}
\end{figure}
\begin{figure}[h]
\centering
\includegraphics[scale=0.45]{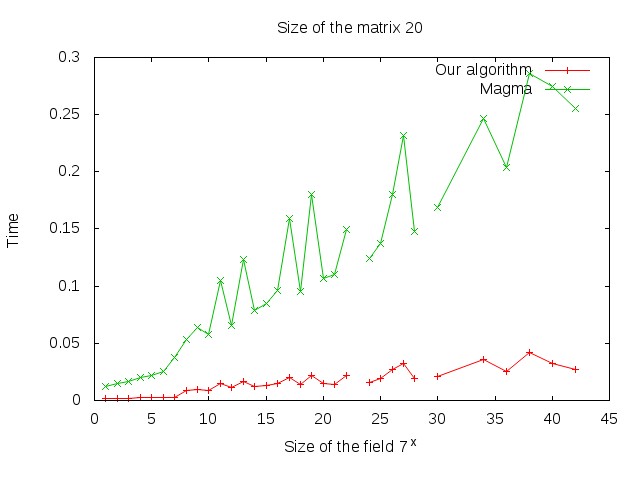}
\includegraphics[scale=0.45]{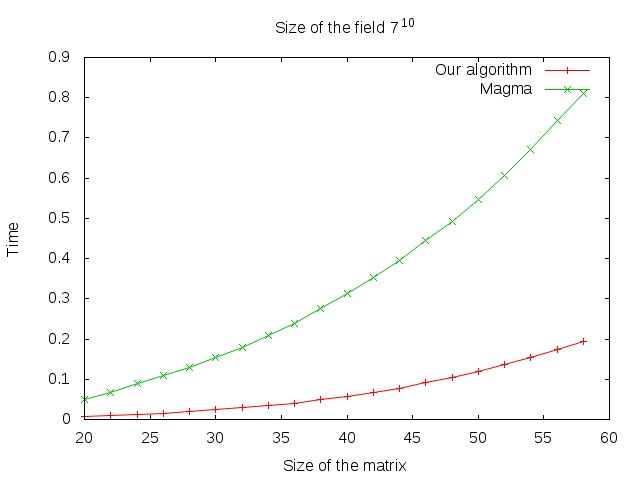}
\caption{Some simulations comparing our algorithm with the one inbuilt in Magma for symplectic groups}
\end{figure}
It seems that our algorithms perform better than that of Costi's on
all fronts.

\bibliographystyle{amsplain}
\bibliography{GCC_algorithm}
\end{document}